\documentclass[11pt,reqno]{amsart}
\usepackage{amssymb,amscd}
\newtheorem{thm}{Theorem}[section]
\newtheorem{prop}[thm]{Proposition}
\newtheorem{lemma}[thm]{Lemma}
\newtheorem{cor}[thm]{Corollary}
\newtheorem{remark}[thm]{Remark}

\newtheorem{definition}[thm]{Definition}

\def\bN{\mathbb{N}}
\def\Prob{\mathrm{Prob}}
\def\bx{\mathbf{x}}
\def\bR{\mathbb{R}}
\def\sym{\mathrm{sym}}
\def\eps{\varepsilon}
\def\cI{\mathcal{I}}
\def\1{\mathbf{1}}
\def\cX{\mathcal{X}}

\begin{document}
\baselineskip=15pt

\title[A new approach to mutual information.\ II]
{A new approach to mutual information.\ II}

\author[F. Hiai]{Fumio Hiai$\,^{1}$}
\address{Graduate School of Information Sciences,
Tohoku University, Aoba-ku, Sendai 980-8579, Japan}
\author[T. Miyamoto]{Takuho Miyamoto}
\address{Graduate School of Information Sciences,
Tohoku University, Aoba-ku, Sendai 980-8579, Japan}

\thanks{$^1\,$Supported in part by Grant-in-Aid for Scientific Research (B)17340043.}


\maketitle

\begin{abstract}
A new concept of mutual pressure is introduced for potential functions on both continuous
and discrete compound spaces via discrete micro-states of permutations, and its relations
with the usual pressure and the mutual information are established. This paper is a
continuation of the paper of Hiai and Petz in Banach Center Publications, Vol.~78.
\end{abstract}

\section*{Introduction}

Entropy and pressure are two basic quantities in statistical physics as well as information
theory, which are in the duality relation via the Legendre transforms of each other. Mutual
information is another important entropic quantity in information theory. The aim of this
paper is to seek for the mutual version of pressure whose Legendre transform is equal to
the mutual information.

The mutual information of two random variables $X$ and $Y$ is defined as the relative
entropy
$$
I(X\wedge Y):=S(\mu_{(X,Y)}\,\|\,\mu_X\otimes\mu_Y),
$$
where $\mu_{(X,Y)}$ is the joint distribution measure of $(X,Y)$ and $\mu_X\otimes\mu_Y$
is the product of the respective distribution measures of $X,Y$. This is also expressed as
$$
I(X\wedge Y)=-S(X,Y)+S(X)+S(Y)
$$
in terms of the Shannon entropy $S(\cdot)$ when $X,Y$ are discrete random variables. When
$X,Y$ are continuous variables, the expression holds with the Boltzmann-Gibbs entropy
$H(\cdot)$ in place of $S(\cdot)$ (as long as $H(X)$ and $H(Y)$ are finite). These
definitions and expressions are naturally extended to the case of more than two random
variables.

In the classical (= commutative) probability setting, we developed in the previous paper
\cite{HP1} a certain ``discretization approach" to the mutual information by using
``discrete micro-states" of permutations. In this paper we apply the same idea to introduce
the notion of the ``mutual pressure" for (continuous) potential functions on compound phase
spaces. We consider the $n$-fold product $[-R,R]^n$ of the bounded interval $[-R,R]$, which
is regarded as the phase space for an $n$-tuple of real bounded random variables. For a
real continuous function $h$ on $[-R,R]^n$ the usual pressure of $h$ is given by
$$
P(h):=\log\int_{[-R,R]^n}e^{h(\bx)}\,d\bx.
$$
For an $n$-tuple $(\mu_1,\dots,\mu_n)$ of probability measures on $[-R,R]$, we choose an
approximating sequence $(\xi_1(N),\dots,\xi_n(N))$ such that $\xi_i(N)$ are vectors in
$[-R,R]_\le^N$ (having the coordinates in increasing order) and $\xi_i(N)\to\mu_i$ in
moments as $N\to\infty$ for $1\le i\le n$. We define the mutual pressure
$P_\sym(h:\mu_1,\dots,\mu_n)$ of $h$ with respect to $(\mu_1,\dots,\mu_n)$ to be the
$\limsup$ as $N\to\infty$ of the asymptotic average
$$
{1\over N}\log\Biggl[{1\over(N!)^n}\sum_{\sigma_1,\dots,\sigma_n\in S_N}
\exp\bigl(N\kappa_N(h(\sigma_1(\xi_1(N)),\dots,\sigma_n(\xi_n(N))))\bigr)\biggr],
$$
over permutations $\sigma_i\in S_N$, where
$\kappa_N(h(\bx_1,\dots,\bx_n)):={1\over N}\sum_{j=1}^Nh(x_{1j},\dots,x_{nj})$ for
$\bx_i=(x_{i1},\dots,x_{iN})\in[-R,R]^N$, $1\le i\le n$ (Definition 2.1). Then the
inequality
$$
P(h)\ge P_\sym(h:\mu_1,\dots,\mu_n)+\sum_{i=1}^nH(\mu_i)
$$
is shown to hold, and the equality case is characterized in a natural way (Theorem 3.2).
Moreover, for a probability measure $\mu$ on $[-R,R]^n$ with marginal measures
$\mu_1,\dots,\mu_n$ on $[-R,R]$, the Legendre transform of $P_\sym(h:\mu_1,\dots,\mu_n)$
is shown to be equal to the mutual information $-H(\mu)+\sum_{i=1}^nH(\mu_i)$ as long as
$H(\mu_i)>-\infty$ for $1\le i\le n$ (Theorem 3.5).

The same approach can be also applied to the setting of discrete phase spaces, when the
Shannon entropy takes the place of the Boltzmann-Gibbs entropy. We deal with the discrete
case in Section 4 separately since the discussions are considerably different from the
continuous case due to the difference of entropies.

\section{Preliminaries in the continuous case}
\setcounter{equation}{0}

Let $R>0$ and $n\in\bN$ be fixed throughout. We denote by $\Prob([-R,R]^n)$ the set of
probability measures on the $n$-fold product $[-R,R]^n$ ($\subset\bR^n$), and by
$C_\bR([-R,R]^n)$ the real Banach space of real continuous functions on $[-R,R]^n$ with
the sup-norm $\|f\|:=\max\{|f(\bx)|:\bx\in[-R,R]^n\}$. The {\it Boltzmann-Gibbs entropy}
of a probability measure $\mu$ on $[-R,R]^n$ is defined to be
$$
H(\mu):=-\int_{[-R,R]^n}p(\bx)\log p(\bx)\,d\bx
$$
if $\mu$ has the joint density $p(\bx)$ with respect to the Lebesgue measure $d\bx$ on
$\bR^N$; otherwise $H(\mu):=-\infty$. A measure $\mu\in\Prob([-R,R]^n)$ typically arises
as the joint distribution of an $n$-tuple $(X_1,\dots,X_n)$ of real random variables
bounded by $R$ (i.e., $|X_i|\le R$) on a probability space. In this case, we have
$H(\mu)=H(X_1,\dots,X_n)$.

For avector $\bx=(x_1,\dots,x_N)$ in $\bR^N$ we write $\|\bx\|_1:=N^{-1}\sum_{j=1}^N|x_j|$.
The mean value of $\bx$ is given by
$$
\kappa_N(\bx):={1\over N}\sum_{j=1}^Nx_j.
$$
For each $N,m\in\bN$ and $\delta>0$ we define $\Delta_R(\mu;N,m,\delta)$ to be the set
of all $n$-tuples $(\bx_1,\dots,\bx_n)$ of $\bx_i=(x_{i1},\dots,x_{iN})\in[-R,R]^N$,
$1\le i\le n$, such that
$$
|\kappa_N(\bx_{i_1}\cdots\bx_{i_k})-\mu(x_{i_1}\cdots x_{i_k})|<\delta
$$
for all $i_1,\dots,i_k\in\{1,\dots,n\}$ with $1\le k\le m$, where
$\bx_{i_1}\cdots\bx_{i_k}$ stands for the pointwise product, i.e.,
$$
\bx_{i_1}\cdots\bx_{i_k}
:=(x_{i_11}\cdots x_{i_k1},x_{i_12}\cdots x_{i_k2},\ \dots\ ,x_{i_1N}\cdots x_{i_kN})
\in\bR^N,
$$
and
$$
\mu(x_{i_1}\cdots x_{i_k})
:=\int_{[-R,R]^n}x_{i_1}\cdots x_{i_k}\,d\mu(x_1,\dots,x_n).
$$
Then it is known \cite[5.1.1]{HP} that the limit
$$
\lim_{N\to\infty}{1\over N}\log\lambda_N^{\otimes n}(\Delta_R(\mu;N,m,\delta))
$$
exists, where $\lambda_N$ stands for the Lebesgue measure on $\bR^N$, and furthermore
we have
$$
H(\mu)=\lim_{m\to\infty,\delta\searrow0}\lim_{N\to\infty}
{1\over N}\log\lambda_N^{\otimes n}(\Delta_R(\mu;N,m,\delta)).
$$
In \cite{HP1} we introduced some kinds of mutual information $I_\sym(\mu)$ and
$\overline I_\sym(\mu)$, and established their relations with $H(\mu)$ as follows.

\begin{definition}\label{D-1.1}{\rm
Let $\mu\in\Prob([-R,R]^n)$ and $\mu_i$ be the restriction (or the marginal) of $\mu$
to the $i$th component $[-R,R]$ of $[-R,R]^n$ for $1\le i\le n$. Choose and fix a
sequence of $n$-tuples $\Xi(N)=(\xi_1(N),\dots,\xi_n(N))$, $N\in\bN$, of
$\bR^N$-vectors $\xi_i(N)$ in
$[-R,R]_\le^N:=\{(x_1,\dots,x_N)\in[-R,R]^N: x_1\le\dots\le x_N\}$ such that
$\kappa_N(\xi_i(N)^k)\to\int x^k\,d\mu_i(x)$ as $N\to\infty$ for all $k\in\bN$, i.e.,
$\xi_i(N)\to\mu_i$ in moments for $1\le i\le n$. We call such a sequence $\Xi(N)$ an
{\it approximating sequence} for $(\mu_1,\dots,\mu_n)$. For $N\in\bN$ the action of
the symmetric group $S_N$ on $\bR^N$ is given by
$$
\sigma(\bx):=(x_{\sigma^{-1}(1)},\dots,x_{\sigma^{-1}(N)})
$$
for $\sigma\in S_N$ and $\bx=(x_1,\dots,x_N)\in\bR^N$. For each $N,m\in\bN$ and
$\delta>0$ we define $\Delta_\sym(\mu:\Xi(N);N,m,\delta)$ to be the set of all
$(\sigma_1,\dots,\sigma_n)\in S_N^n$ such that
$$
(\sigma_1(\xi_1(N)),\dots,\sigma_n(\xi_n(N)))\in\Delta_R(\mu;N,m,\delta).
$$
We define
$$
I_\sym(\mu):=-\lim_{m\to\infty,\delta\searrow0}\limsup_{N\to\infty}
{1\over N}\log\gamma_{S_N}^{\otimes n}(\Delta_\sym(\mu:\Xi(N);N,m,\delta)),
$$
where $\gamma_{S_N}$ is the uniform probability measure on $S_N$, and define also
$\overline I_\sym(\mu)$ by replacing $\limsup$ by $\liminf$. This definitions of
$I_\sym(\mu)$ and $\overline I_\sym(\mu)$ are independent of the choice of an
approximating sequence $\Xi(N)$ for $(\mu_1,\dots,\mu_n)$ (\cite[Lemma 1.5]{HP1}).
}\end{definition}

\begin{thm}\label{T-1.2}{\rm(\cite[Theorem 1.6]{HP1})}\quad
For every $\mu\in\Prob([-R,R]^n)$ with marginals
$\mu_1,\dots,\allowbreak\mu_n\in\Prob([-R,R])$,
$$
H(\mu)=-I_\sym(\mu)+\sum_{i=1}^nH(\mu_i)
=-\overline I_\sym(\mu)+\sum_{i=1}^nH(\mu_i).
$$
\end{thm}

The {\it pressure} of $h\in C_\bR([-R,R]^n)$ is given by
$$
P(h):=\log\int_{[-R,R]^n}e^{h(\bx)}\,d\bx.
$$
It is well known that the pressure function $P(h)$ for $h\in C_\bR([-R,R]^n)$ and the
(minus) Boltzmann-Gibbs entropy $-H(\mu)$ for $\mu\in\Prob([-R,R]^n)$ are in the duality
relation in the sense that they are the Legendre transforms of each other. That is,
\begin{align}
H(\mu)&=\inf\{-\mu(h)+P(h):h\in C_\bR([-R,R]^n)\},
\quad\mu\in\Prob([-R,R]^n), \label{F-1.1}\\
P(h)&=\max\{\mu(h)+H(\mu):\mu\in\Prob([-R,R]^n)\},
\quad h\in C_\bR([-R,R]^n). \nonumber
\end{align}
Furthermore, for every $h\in C_\bR([-R,R]^n)$ the {\it Gibbs probability measure}
$\mu_h$ associated with $h$ is given by
$$
d\mu_h(\bx):={1\over Z_h}\,e^{h(\bx)}\,d\bx\quad\mbox{with}
\quad Z_h:=\int_{[-R,R]^n}e^{h(\bx)}\,d\bx=e^{P(h)},
$$
which is characterized by the variational equality
$$
P(h)=\mu_h(h)+H(\mu_h),
$$
that is, $\mu_h$ is a unique maximizer of $\mu\in\Prob([-R,R]^n)\mapsto\mu(h)+H(\mu)$.

\section{Mutual pressure and its Legendre transform}
\setcounter{equation}{0}

In the setting of continuous compound spaces described in Section 1, we introduce the
mutual version of pressure for continuous potential functions, and consider its Legendre
transform that is a version of the mutual information.

\begin{definition}\label{D-2.1}{\rm
Let $\mu_1,\dots,\mu_n\in\Prob([-R,R])$ be given and choose an approximating sequence
$\Xi(N)=(\xi_1(N),\dots,\xi_n(N))$ of $\xi_i(N)\in[-R,R]_\le^N$ for
$(\mu_1,\dots,\mu_n)$ as in Definition \ref{D-1.1}. For each $h\in C_\bR([-R,R]^n)$ and
$\bx_i=(x_{i1},\dots,x_{iN})\in[-R,R]^N$, $1\le i\le n$, define
\begin{equation}\label{F-2.1}
h(\bx_1,\dots,\bx_n)
:=(h(x_{11},\dots,x_{n1}),h(x_{12},\dots,x_{n2}),\dots,h(x_{1N},\dots,x_{nN}))
\in\bR^N
\end{equation}
and hence
\begin{equation}\label{F-2.2}
\kappa_N(h(\bx_1,\dots,\bx_n))
:={1\over N}\sum_{j=1}^Nh(x_{1j},\dots,x_{nj}).
\end{equation}
For each $h\in C_\bR([-R,R]^n)$ we define the {\it mutual pressure} of $h$ with respect
to $(\mu_1,\dots,\mu_n)$ to be
\begin{align*}
&P_\sym(h:\mu_1,\dots,\mu_n) \\
&\quad:=\limsup_{N\to\infty}{1\over N}\log\int_{S_N^n}
\exp\bigl(N\kappa_N(h(\sigma_1(\xi_1(N)),\dots,\sigma_n(\xi_n(N))))\bigr)
\,d\gamma_{S_N}^{\otimes n}(\sigma_1,\dots,\sigma_n) \\
&\quad\ =\limsup_{N\to\infty}{1\over N}\log\Biggl[
{1\over(N!)^n}\sum_{\sigma_1,\dots,\sigma_n\in S_N}
\exp\bigl(N\kappa_N(h(\sigma_1(\xi_1(N)),\dots,\sigma_n(\xi_n(N))))\bigr)\Biggr].
\end{align*}
}\end{definition}

The above definition is justified by the following:

\begin{lemma}\label{L-1.2}
$P_\sym(h:\mu_1,\dots,\mu_n)$ is independent of the choice of an approximating sequence
$\Xi(N)$ for $(\mu_1,\dots,\mu_n)$.
\end{lemma}

\begin{proof}
Let $\Xi'(N)=(\xi_1'(N),\dots,\xi_n'(N))$ be another approximating sequence for
$(\mu_1,\dots,\mu_n)$. We write $P_\sym(h:\Xi)$ and $P_\sym(h:\Xi')$ for
$P_\sym(h:\mu_1,\dots,\mu_n)$ defined in Definition \ref{D-2.1} with $\Xi(N)$ and
$\Xi'(N)$, respectively. Since $P_\sym(h:\Xi)$ and $P_\sym(h:\Xi')$ are continuous
in $h$ in the norm (see Proposition \ref{P-2.3}\,(3) below), it suffices to prove that
$P_\sym(p:\Xi)=P_\sym(p:\Xi')$ for any real polynomial $p$ of $n$ variables
$x_1,\dots,x_n$. Since $\xi_i(N),\xi_i'(N)\in[-R,R]_\le^N$, for any $\eps>0$ there
exists a $\delta>0$ such that, for every $N\in\bN$, if
$\|\xi_i(N)-\xi_i'(N)\|_1<\delta$ for all $i=1,\dots,n$, then
$$
|\kappa_N(p(\sigma_1(\xi_1(N)),\dots,\sigma_n(\xi_n(N))))
-\kappa_N(p(\sigma_1(\xi_1'(N)),\dots,\sigma_n(\xi_n'(N))))|<\eps
$$
for all $(\sigma_1,\dots,\sigma_n)\in S_N^n$. Thanks to \cite[Lemma 4.3]{V2} (also
\cite[4.3.4]{HP}), there exists an $N_0\in\bN$ such that if $N\ge N_0$ then
$\|\xi_i(N)-\xi_i'(N)\|_1<\delta$ for all $i=1,\dots,n$. Hence we have for every $N\ge N_0$
\begin{align*}
&\Bigg|{1\over N}\log\Biggl[{1\over(N!)^n}\sum_{\sigma_1,\dots,\sigma_n\in S_N}
\exp\bigl(N\kappa_N(p(\sigma_1(\xi_1(N)),\dots,\sigma_n(\xi_n(N))))\bigr)\Biggr] \\
&\qquad-{1\over N}\log\Biggl[{1\over(N!)^n}\sum_{\sigma_1,\dots,\sigma_n\in S_N}
\exp\bigl(N\kappa_N(p(\sigma_1(\xi_1'(N)),\dots,\sigma_n(\xi_n'(N))))\bigr)\Biggr]
\Bigg|<\eps.
\end{align*}
This implies that $|P_\sym(p:\Xi)-P_\sym(p:\Xi')|\le\eps$. Since $\eps>0$ is arbitrary,
the desired conclusion follows.
\end{proof}

The following are basic properties of $P_\sym(h:\mu_1,\dots,\mu_n)$, whose proofs are
straightforward.

\begin{prop}\label{P-2.3}
Let $\mu_1,\dots,\mu_n\in\Prob([-R,R])$.
\begin{itemize}
\item[(1)] When $n=1$, $P_\sym(h:\mu_1)=\mu_1(h)$ for all $h\in C_\bR([-R,R])$.
\item[(2)] $P_\sym(h:\mu_1,\dots,\mu_n)$ is a convex and increasing function on
$C_\bR([-R,R]^n)$.
\item[(3)] $|P_\sym(h:\mu_1,\dots,\mu_n)-P_\sym(h':\mu_1,\dots,\mu_n)|\le\|h-h'\|$
for all $h,h'\in C_\bR([-R,R]^n)$.
\item[(4)] If $1\le m<n$, $h^{(1)}\in C_\bR([-R,R]^m)$,
$h^{(2)}\in C_\bR([-R,R]^{n-m})$ and
$h(x_1,\dots,x_n):=h^{(1)}(x_1,\dots,x_m)+h^{(2)}(x_{m+1},\dots,x_n)$, then
$$
P_\sym(h:\mu_1,\dots,\mu_n)
\le P_\sym(h^{(1)}:\mu_1,\dots,\mu_m)+P_\sym(h^{(2)}:\mu_{m+1},\dots,\mu_n).
$$
\end{itemize}
\end{prop}

\begin{definition}\label{D-2.4}{\rm
Let $\mu\in\Prob([-R,R]^n)$ with marginals $\mu_1,\dots,\mu_n\in\Prob([-R,R])$.
Define
$$
\cI_\sym(\mu):=\sup\{\mu(h)-P_\sym(h:\mu_1,\dots,\mu_n):h\in C_\bR([-R,R]^n)\},
$$
that is, $\cI_\sym(\mu)$ is the Legendre transform of $P_\sym(h:\mu_1,\dots,\mu_n)$.
Furthermore, we say that $\mu$ is {\it mutually equilibrium} associated with
$h\in C_\bR([-R,R]^n)$ if the variational equality
$$
\cI_\sym(\mu)=\mu(h)-P_\sym(h:\mu_1,\dots,\mu_n)
$$
holds.
}\end{definition}

The next proposition says that $P_\sym(h:\mu_1,\dots,\mu_n)$ is the converse Legendre
transform of $\cI_\sym(\mu)$.

\begin{prop}\label{P-2.5}
For every $h\in C_\bR([-R,R]^n)$ and $\mu_1,\dots,\mu_n\in\Prob([-R,R])$,
$$
P_\sym(h:\mu_1,\dots,\mu_n)
=\max\{\mu(h)-\cI_\sym(\mu):\mu\in\Prob_{\mu_1,\dots,\mu_n}([-R,R]^n)\},
$$
where $\Prob_{\mu_1,\dots,\mu_n}([-R,R]^n)$ is the set of all $\mu\in\Prob([-R,R]^n)$
whose restriction to the $i$th component of $[-R,R]^n$ is $\mu_i$ for $1\le i\le n$.
Hence there exists a mutually equilibrium probability measure associated with $h$
whose marginals are $\mu_1,\dots,\mu_n$.
\end{prop}

\begin{proof}
One can consider $\Prob([-R,R]^n)$ as a closed convex subset of the dual (real) Banach
space $C_\bR([-R,R]^n)^*$ of $C_\bR([-R,R]^n)$. Let
$F:C_\bR([-R,R]^n)^*\to(-\infty,+\infty]$ be the conjugate (or the Legendre transform)
of $P_\sym(h:\mu_1,\dots,\mu_n)$, i.e.,
$$
F(\psi):=\sup\{\psi(h)-P_\sym(h:\mu_1,\dots,\mu_n):h\in C_\bR([-R,R]^n)\}
$$
for $\psi\in C_\bR([-R,R]^n)^*$. We then prove that
\begin{equation}\label{F-2.3}
\begin{cases}
F(\mu)=\cI_\sym(\mu) & \text{if $\mu\in\Prob_{\mu_1,\dots,\mu_n}([-R,R]^n)$}, \\
F(\psi)=+\infty &
\text{if $\psi\in C_\bR([-R,R]^n)^*\setminus\Prob_{\mu_1,\dots,\mu_n}([-R,R]^n)$}.
\end{cases}
\end{equation}
The first equality is just the definition of $\cI_\sym(\mu)$. The second follows from
the following three claims.

(a)\enspace If $\psi(h)<0$ for some $h\in C_\bR([-R,R]^n)$ with $h\ge0$, then
$F(\psi)=+\infty$. In fact, for $\alpha<0$ we have
$P_\sym(\alpha h:\mu_1,\dots,\mu_n)\le P_\sym(0:\mu_1,\dots,\mu_n)=0$ by Proposition
\ref{P-2.3}\,(2) so that
$$
\psi(\alpha h)-P_\sym(\alpha h:\mu_1,\dots,\mu_n)
\ge\alpha\psi(h)\longrightarrow+\infty
$$
as $\alpha\to-\infty$.

(b)\enspace If $\psi(\1)\ne1$, then $F(\psi)=+\infty$. In fact, since
$P_\sym(\alpha\1:\mu_1,\dots,\mu_n)=\alpha$ for $\alpha\in\bR$, it follows that
$$
\psi(\alpha h)-P_\sym(\alpha\1:\mu_1,\dots,\mu_n)
=\alpha(\psi(\1)-1)\longrightarrow+\infty
$$
as $\alpha\to+\infty$ or $-\infty$ accordingly as $\psi(\1)<1$ or $\psi(\1)>1$.

(c)\enspace Assume that $\mu\in\Prob([-R,R]^n)$ but
$\mu\not\in\Prob_{\mu_1,\dots,\mu_n}([-R,R]^n)$. Then there exists an
$f\in C_\bR([-R,R])$ such that $\mu(f^{(i)})>\mu_i(f)$ for some $1\le i\le n$, where
$f^{(i)}(\bx):=f(x_i)$ for $\bx=(x_1,\dots,x_n)\in[-R,R]^n$. Since
$$
P_\sym(\alpha f^{(i)}:\mu_1,\dots,\mu_n)
=\lim_{N\to\infty}\alpha f(\xi_i(N))=\alpha\mu_i(f)
$$
for $\alpha\in\bR$, it follows that
$$
\mu(\alpha f^{(i)})-P_\sym(\alpha f^{(i)}:\mu_1,\dots,\mu_n)
=\alpha(\mu(f^{(i)})-\mu_i(f))\longrightarrow+\infty
$$
as $\alpha\to+\infty$.

Hence \eqref{F-2.3} is proved. Since $P_\sym(h:\mu_1,\dots,\mu_n)$ is a convex
continuous function on $C_\bR([-R,R]^n)$ by Proposition \ref{P-2.3}, the duality
theorem for conjugate functions implies that
\begin{align*}
P_\sym(h:\mu_1,\dots,\mu_n)
&=\sup\{\psi(h)-F(\psi):\psi\in C_\bR([-R,R]^n)^*\} \\
&=\sup\{\mu(h)-\cI_\sym(\mu):\mu\in\Prob_{\mu_1,\dots,\mu_n}([-R,R]^n)\}.
\end{align*}
Since $\Prob_{\mu_1,\dots,\mu_n}([-R,R]^n)$ is weakly* compact and $\cI_\sym(\mu)$ is
weakly* lower semicontinuous on $\Prob_{\mu_1,\dots,\mu_n}([-R,R]^n)$, the above latter
supremum is attained by some $\mu\in\Prob_{\mu_1,\dots,\mu_n}([-R,R]^n)$.
\end{proof}

\begin{prop}\label{P-2.6}
The function $P_\sym(h:\mu_1,\dots,\mu_n)$ is jointly continuous on
$C_\bR([-R,R]^n)\times(\Prob([-R,R]))^n$ with respect to the norm topology on
$C_\bR([-R,R]^n)$ and the weak* topology on $\Prob([-R,R])$.
\end{prop}

\begin{proof}
Let $h,h'\in C_\bR([-R,R]^n)$ and $\mu_i,\mu_i'\in\Prob([-R,R])$, $1\le i\le n$. For
any $\eps>0$ choose a real polynomial $p$ of $n$ variables $x_1,\dots,x_n$ such that
$\|p-h\|<\eps$. We have
\begin{align*}
&|P_\sym(h:\mu_1,\dots,\mu_n)-P_\sym(h':\mu_1',\dots,\mu_n')| \\
&\qquad\le|P_\sym(h:\mu_1,\dots,\mu_n)-P_\sym(p:\mu_1,\dots,\mu_n)| \\
&\qquad\qquad+|P_\sym(p:\mu_1,\dots,\mu_n)-P_\sym(p:\mu_1',\dots,\mu_n')| \\
&\qquad\qquad+|P_\sym(p:\mu_1',\dots,\mu_n')-P_\sym(h':\mu_1',\dots,\mu_n')| \\
&\qquad\le\|h-p\|+\|p-h'\|
+|P_\sym(p:\mu_1,\dots,\mu_n)-P_\sym(p:\mu_1',\dots,\mu_n')| \\
&\qquad\le2\eps+\|h-h'\|
+|P_\sym(p:\mu_1,\dots,\mu_n)-P_\sym(p:\mu_1',\dots,\mu_n')|
\end{align*}
by Proposition \ref{P-2.3}\,(3). Recall that the weak* topology on $\Prob([-R,R])$ is
metrizable with the metric
$\rho(\nu,\nu'):=\sum_{k=1}^\infty(2R)^{-k}|\nu(x^k)-\nu'(x^k)|$, where
$\nu(x^k):=\int x^k\,d\nu(x)$. It suffices to show that there exists a $\delta>0$ such
that if $\rho(\mu_i,\mu_i')<\delta$ for $1\le i\le n$, then
$$
|P_\sym(p:\mu_1,\dots,\mu_n)-P_\sym(p:\mu_1',\dots,\mu_n')|\le\eps.
$$
One can choose a $\delta_1>0$ such that, for every $N\in\bN$, if
$\bx_i,\bx_i'\in[-R,R]_\le^N$ and $\|\bx_i-\bx_i'\|_1<\delta_1$ for
$1\le i\le n$, then
$$
|\kappa_N(p(\sigma_1(\bx_1),\dots,\sigma_n(\bx_n)))
-\kappa_N(p(\sigma_1(\bx_1'),\dots,\sigma_n(\bx_n')))|<\eps
$$
for all $(\sigma_1,\dots,\sigma_n)\in S_N^n$. Thanks to \cite[Lemma 4.3]{V2} one can
choose an $m\in\bN$ and a $\delta_2>0$ such that, for every $N\in\bN$, if
$\bx,\bx'\in[-R,R]_\le^N$ and $|\kappa_N(\bx^k)-\kappa_N(\bx^{\prime k})|<\delta_2$
for all $k=1,\dots,m$, then $\|\bx-\bx'\|_1<\delta_1$. Then choose a $\delta_3>0$ such
that if $\nu,\nu'\in\Prob([-R,R])$ and $\rho(\nu,\nu')<\delta_3$, then
$|\nu(x^k)-\nu'(x^k)|<\delta_2/2$ for all $k=1,\dots,m$. Now assume that
$\mu_i,\mu_i'\in\Prob([-R,R])$ and $\rho(\mu_i,\mu_i')<\delta_3$ for $1\le i\le n$.
Let $\Xi(N)=(\xi_1(N),\dots,\xi_n(N))$ and $\Xi'(N)=(\xi_1'(N),\dots,\xi_n'(N))$ be
approximating sequences for $(\mu_1,\dots,\mu_n)$ and $(\mu_1',\dots,\mu_n')$,
respectively, with $\xi_i(N),\xi_i'(N)\in[-R,R]_\le^N$. There exists an $N_0\in\bN$
such that if $N\ge N_0$ then for $1\le i\le n$ we have
\begin{align*}
&|\kappa_N(\xi_i(N)^k)-\kappa_N(\xi_i'(N)^k)| \\
&\qquad\le|\kappa_N(\xi_i(N)^k)-\mu_i(x^k)|+|\mu_i(x^k)-\mu_i'(x^k)|
+|\mu_i'(x^k)-\kappa_n(\xi_i'(N)^k)|<\delta_2
\end{align*}
for all $k=1,\dots,m$ so that $\|\xi_i(N)-\xi_i'(N)\|_1<\delta_1$. Hence if $N\ge N_0$
then we have
$$
|\kappa_N(p(\sigma_1(\xi_1(N)),\dots,\sigma_n(\xi_n(N))))
-\kappa_N(p(\sigma_1(\xi_1'(N)),\dots,\sigma_n(\xi_n'(N))))|<\eps
$$
for all $(\sigma_1,\dots,\sigma_n)\in S_N^n$. This implies that
$|P_\sym(p:\mu_1,\dots,\mu_n)-P_\sym(p:\mu_1',\dots,\mu_n')|\le\eps$, as required.
\end{proof}

\begin{cor}\label{C-2.7}
The function $\cI_\sym(\mu)$ is weakly* lower semicontinuous on $\Prob([-R,R]^n)$.
\end{cor}

\begin{proof}
Let $\mu$ and $\mu^{(k)}$, $k\in\bN$, be in $\Prob([-R,R]^n)$ such that
$\mu^{(k)}\to\mu$ weakly*. Let $\mu_i$ and $\mu_i^{(k)}$, $1\le i\le n$, be the
marginals of $\mu$ and $\mu^{(k)}$, respectively. Since $\mu_i^{(k)}\to\mu_i$ weakly*
as $k\to\infty$ for $1\le i\le n$, Proposition \ref{P-2.6} implies that for every
$h\in C_\bR([-R,R]^n)$
\begin{align*}
\mu(h)-P_\sym(h:\mu_1,\dots,\mu_n)
&=\lim_{k\to\infty}\{\mu^{(k)}(h)-P_\sym(h:\mu_1^{(k)},\dots,\mu_n^{(k)})\} \\
&\le\liminf_{k\to\infty}\cI_\sym(\mu^{(k)})
\end{align*}
so that $\cI_\sym(\mu)\le\liminf_{k\to\infty}\cI_\sym(\mu^{(k)})$, as required.
\end{proof}

\section{Relations of $P_\sym(h)$ with $P(h)$ and of $\cI_\sym(\mu)$ with $H(\mu)$}
\setcounter{equation}{0}

First let us recall the Sanov large deviation in the form suitable for our purpose. Let
$h_0\in C_\bR([-R,R])$ and $\mu_0$ be the Gibbs probability measure associated with $h_0$,
i.e.,
$$
d\mu_0(x):={1\over Z_{h_0}}\,e^{h_0(x)}\,dx
\quad\mbox{with}\quad Z_{h_0}:=\int_{[-R,R]}e^{h_0(x)}\,dx.
$$
Consider the infinite product probability space $([-R,R]^\infty,\mu_0^{\otimes\infty})$
and i.i.d.\ (independent and identically distributed) random variables $x_1,x_2,\dots$
consisting of coordinate variables of $[-R,R]^\infty$. The {\it Sanov theorem} (see
\cite[6.2.10]{DZ}) says that the empirical measure (random probability measure)
$$
{\delta_{x_1}+\dots+\delta_{x_N}\over N}
$$
satisfies the large deviation principle in the scale $1/N$ with the good rate function
$S(\mu\,\|\,\mu_0)$ for $\mu\in\Prob([-R,R])$, where $S(\mu\,\|\,\mu_0)$ denotes the
{\it relative entropy} (or the {\it Kullback-Leibler divergence}) of $\mu$ with respect
to $\mu_0$. That is,
\begin{align*}
\limsup_{N\to\infty}{1\over N}\log\mu_0^{\otimes N}
\biggl({\delta_{x_1}+\dots+\delta_{x_N}\over N}\in F\biggr)
&\le-\inf\{S(\mu\,\|\,\mu_0):\mu\in F\}, \\
\liminf_{N\to\infty}{1\over N}\log\mu_0^{\otimes N}
\biggl({\delta_{x_1}+\dots+\delta_{x_N}\over N}\in G\biggr)
&\ge-\inf\{S(\mu\,\|\,\mu_0):\mu\in G\}
\end{align*}
for every closed subset $F$ and every open subset $G$ of $\Prob([-R,R])$ in the weak*
topology. As remarked in \cite[p.\,211]{HP}, it then follows (based on the Borel-Cantelli
lemma) that the empirical measure $(\delta_{x_1}+\dots+\delta_{x_N})/N$ converges to
$\mu_0$ in the weak* topology almost surely. In the next lemma we state some consequences
of the above large deviation, which will play a crucial role in our later discussions.

\begin{lemma}\label{L-3.1}
Let $h_0$ and $\mu_0$ be as above. Then:
\begin{itemize}
\item[(a)] For every $m\in\bN$ and $\delta>0$,
$$
\lim_{N\to\infty}\mu_0^{\otimes N}(\Delta_R(\mu_0;N,m,\delta))=1.
$$
\item[(b)] If $\mu_1\in\Prob([-R,R])$ and $\mu_1\ne\mu_0$, then there exist an
$m\in\bN$ and a $\delta>0$ such that
$$
\limsup_{N\to\infty}{1\over N}\log\mu_0^{\otimes N}(\Delta_R(\mu_1;N,m,\delta))<0.
$$
\end{itemize}
\end{lemma}

\begin{proof}
(a)\enspace For each $m\in\bN$ and $\delta>0$ set
$$
G(\mu_0;m,\delta):=\{\mu\in\Prob([-R,R]):
|\mu(x^k)-\mu_0(x^k)|<\delta,\ 1\le k\le m\},
$$
which is a weak* neighborhood of $\mu_0$. Note that
$\bx=(x_1,\dots,x_N)\in\Delta_R(\mu_0;N,m,\delta)$ is equivalent to
$(\delta_{x_1}+\dots+\delta_{x_N})/N\in G(\mu_0;m,\delta)$. Since 
$(\delta_{x_1}+\dots+\delta_{x_N})/N\to\mu_0$ weakly* in the sense of almost sure
(with respect to $\mu_0^{\otimes\infty}$) as remarked above, we have
$$
\mu_0^{\otimes N}(\Delta_R(\mu_0;N,m,\delta))
=\mu_0^{\otimes N}\biggl({\delta_{x_1}+\dots+\delta_{x_N}\over N}
\in G(\mu_0;m,\delta)\biggr)\longrightarrow1
$$
as $N\to\infty$.

(b)\enspace Let $\mu_1\in\Prob([-R,R])$ with $\mu_1\ne\mu_0$. One can find an $m\in\bN$
and a $\delta>0$ so that the weak* closed subset
$$
F(\mu_1;m,\delta):=\{\mu\in\Prob([-R,R]):
|\mu(t^k)-\mu_1(t^k)|\le\delta,\ 1\le k\le m\}
$$
does not contain $\mu_0$. The large deviation principle implies that
\begin{align*}
&\limsup_{N\to\infty}{1\over N}
\log\mu_0^{\otimes N}(\Delta_R(\mu_1;N,m,\delta)) \\
&\qquad\le\limsup_{N\to\infty}{1\over N}
\log\mu_0^{\otimes N}\biggl({\delta_{x_1}+\dots+\delta_{x_N}\over N}
\in F(\mu_1;m,\delta)\biggr) \\
&\qquad\le-\inf\{S(\mu\,\|\,\mu_0):\mu\in F(\mu_1;m,\delta)\}<0,
\end{align*}
because $S(\mu\,\|\,\mu_0)$ is weakly* lower semicontinuous and so attains the minimum
($>0$) on a weakly* compact subset $F(\mu_1;m,\delta)$.
\end{proof}

The next theorem gives an exact relation between $P_\sym(h)$ and $P(h)$.

\begin{thm}\label{T-3.2}
For every $h\in C_\bR([-R,R]^n)$ and every $\mu_1,\dots,\mu_n\in\Prob([-R,R])$,
\begin{equation}\label{F-3.1}
P(h)\ge P_\sym(h:\mu_1,\dots,\mu_n)+\sum_{i=1}^nH(\mu_i).
\end{equation}
Moreover the following conditions are equivalent:
\begin{itemize}
\item[(i)] $P(h)=P_\sym(h:\mu_1,\dots,\mu_n)+\sum_{i=1}^nH(\mu_i)$;
\item[(ii)] $\mu_1,\dots,\mu_n$ are the marginals of the Gibbs measure associated with
$h$;
\item[(iii)] for each $i=1,\dots,n$, $\mu_i$ is the Gibbs measure associated with
$h_i\in C_\bR([-R,R])$ defined by
$$
h_i(x):=\log\int_{[-R,R]^{n-1}}e^{h(x_1,\dots,x_{i-1},x,x_{i+1},\dots,x_n)}
\,dx_1\cdots dx_{i-1}dx_{i+1}\cdots dx_n
$$
for $x\in[-R,R]$.
\end{itemize}
\end{thm}

\begin{proof}
Consider the Gibbs probability measure $\mu_h:=Z_h^{-1}e^{h(\bx)}\,d\bx$ associated
with $h$ so that $P(h)=\log Z_h$. Let $N,m\in\bN$ and $\delta>0$. Then it is
straightforward to see that
$$
Z_h^N\mu_h^{\otimes N}\Biggl(\prod_{i=1}^n\Delta_R(\mu_i;N,m,\delta)\Biggr)
=\int_{\prod_{i=1}^n\Delta_R(\mu_i;N,m,\delta)}
\exp\bigl(N\kappa_N(h(\bx_1,\dots,\bx_n))\bigr)\,\prod_{i=1}^nd\bx_i,
$$
where $\prod_{i=1}^n\Delta_R(\mu_i;N,m,\delta)$ in the left-hand side is regarded as a
subset of $(\bR^n)^N$ by the correspondence $(\bx_1,\dots,\bx_n)\leftrightarrow
\bigl((x_{i1})_{i=1}^n,(x_{i2})_{i=1}^n,\dots,(x_{iN})_{i=1}^n\bigr)$
for $\bx_i=(x_{i1},\dots,x_{iN})$. Hence we have
\begin{align}
&Z_h^N\mu_h^{\otimes N}\Biggl(
\prod_{i=1}^n\Delta_R(\mu_i;N,m,\delta)\Biggr) \nonumber\\
&\qquad=(N!)^n\int_{\prod_{i=1}^n\bigl(\Delta_R(\mu_i;N,m,\delta)\cap\bR_\le^N\bigr)}
{1\over(N!)^n} \nonumber\\
&\hskip3cm\times\sum_{\sigma_1,\dots,\sigma_n\in S_N}
\exp\bigl(N\kappa_N(h(\sigma_1(\bx_1),\dots,\sigma_n(\bx_n)))\bigr)
\,\prod_{i=1}^nd\bx_i. \label{F-3.2}
\end{align}
Let $\Xi(N)=(\xi_1(N),\dots,\xi_n(N))$ be an approximating sequence for
$(\mu_1,\dots,\mu_n)$. For any $\eps>0$ there exists a real polynomial $p$ of variables
$x_1,\dots,x_n$ such that $\|p-h\|<\eps$. Then there exist an $m\in\bN$, a $\delta>0$
and an $N_0\in\bN$ such that, for every $N\ge N_0$, if
$\bx_i\in\Delta_R(\mu_i;N,m,\delta)\cap\bR_\le^N$ for $1\le i\le n$, then we have
$$
|\kappa_N(p(\sigma_1(\bx_1),\dots,\sigma_n(\bx_n))
-\kappa_N(p(\sigma_1(\xi_1(N)),\dots,\sigma_n(\xi_n(N))))|<\eps
$$
so that
$$
|\kappa_N(h(\sigma_1(\bx_1),\dots,\sigma_n(\bx_n))
-\kappa_N(h(\sigma_1(\xi_1(N)),\dots,\sigma_n(\xi_n(N))))|<3\eps
$$
for all $\sigma_1,\dots,\sigma_n\in S_N$. Hence by \eqref{F-3.2} we obtain
\begin{align}
Z_h^N&\ge Z_h^N\mu_h^{\otimes N}\Biggl(
\prod_{i=1}^n\Delta_R(\mu_i;N,m,\delta)\Biggr) \nonumber\\
&\ge e^{-3N\eps}{1\over(N!)^n}\sum_{\sigma_1,\dots,\sigma_n\in S_N}\exp
\bigl(N\kappa_N(h(\sigma_1(\xi_1(N)),\dots,\sigma_n(\xi_n(N))))\bigr) \nonumber\\
&\hskip2.5cm\times\prod_{i=1}^n\lambda_N(\Delta_R(\mu_i;N,m,\delta)) \label{F-3.3}
\end{align}
and
\begin{align}
&Z_h^N\mu_h^{\otimes N}\Biggl(
\prod_{i=1}^n\Delta_R(\mu_i;N,m,\delta)\Biggr) \nonumber\\
&\qquad\le e^{3N\eps}{1\over(N!)^n}\sum_{\sigma_1,\dots,\sigma_n\in S_N}\exp
\bigl(N\kappa_N(h(\sigma_1(\xi_1(N)),\dots,\sigma_n(\xi_n(N))))\bigr) \nonumber\\
&\hskip3cm\times\prod_{i=1}^n\lambda_N(\Delta_R(\mu_i;N,m,\delta)). \label{F-3.4}
\end{align}
It follows from \eqref{F-3.3} that
\begin{align*}
P(h)&={1\over N}\log Z_h^N \\
&\ge-3\eps+{1\over N}\log\Biggl[{1\over(N!)^n}
\sum_{\sigma_1,\dots,\sigma_n\in S_N}\exp\bigl(N\kappa_N(h(\sigma_1(\xi_1(N)),
\dots,\sigma_n(\xi_n(N))))\bigr)\Biggr] \\
&\qquad+\sum_{i=1}^n{1\over N}\log\lambda_N(\Delta_R(\mu_i;N,m,\delta)).
\end{align*}
This yields
$$
P(h)\ge-3\eps+P_\sym(h:\mu_1,\dots,\mu_n)
+\sum_{i=1}^n\lim_{N\to\infty}{1\over N}
\log\lambda_N(\Delta_R(\mu_i;N,m,\delta))
$$
thanks to the existence of the limits in the last term. Letting $m\to\infty$ and
$\delta\searrow0$ gives
$$
P(h)\ge-3\eps+P_\sym(h:\mu_1,\dots,\mu_n)+\sum_{i=1}^nH(\mu_i),
$$
which implies inequality \eqref{F-3.1} since $\eps>0$ is arbitrary.

Next let us prove the equivalence of conditions (i)--(iii). For $1\le i\le n$ let
$\mu_{h,i}$ be the $i$th marginal of $\mu_h$. Since
\begin{align}
d\mu_{h,i}(x)
&={1\over Z_h}\Biggl(\int_{[-R,R]^{n-1}}
e^{h(x_1,\dots,x_{i-1},x,x_{i+1}\dots,x_n)}
\,dx_1\cdots dx_{i-1}dx_{i+1}\cdots dx_n\Biggr)\,dx \nonumber\\
&={1\over Z_h}\,e^{h_i(x)}\,dx, \label{F-3.5}
\end{align}
we notice that $\mu_{h,i}$ is the Gibbs measure associated with $h_i$ for $1\le i\le n$.
Hence (ii) $\Leftrightarrow$ (iii) follows. To prove (ii) $\Rightarrow$ (i), assume that
$\mu_i=\mu_{h,i}$ for all $i=1,\dots,n$. Then Lemma \ref{L-3.1}\,(a) gives
\begin{align*}
&\lim_{N\to\infty}\mu_h^{\otimes N}\bigl(\{(\bx_1,\dots,\bx_n)\in([-R,R]^N)^n:
\bx_i\in\Delta_R(\mu_i;N,m,\delta)\}\bigr) \\
&\qquad=\lim_{N\to\infty}\mu_i^{\otimes N}(\Delta_R(\mu_i;N,m,\delta))=1.
\end{align*}
Therefore,
\begin{align*}
&\lim_{N\to\infty}\mu_h^{\otimes N}\Biggl(
\prod_{i=1}^n\Delta_R(\mu_i;N,m,\delta)\Biggr) \\
&\qquad=\lim_{N\to\infty}\mu_h^{\otimes N}\Biggl(
\bigcap_{i=1}^n\{(\bx_1,\dots,\bx_n):\bx_i\in\Delta_R(\mu_i;N,m,\delta)\}
\Biggr)=1.
\end{align*}
Hence it follows from \eqref{F-3.4} that
$$
P(h)\le3\eps+P_\sym(h:\mu_1,\dots,\mu_n)+\sum_{i=1}^nH(\mu_i),
$$
which implies equality in (i).

Conversely, assume (i). Since \eqref{F-3.3} implies
that
\begin{align*}
&P(h)+\limsup_{N\to\infty}{1\over N}\log\mu_h^{\otimes N}
\Biggl(\prod_{i=1}^n\Delta_R(\mu_i;N,m,\delta)\Biggr) \\
&\qquad\ge-3\eps+P_\sym(h:\mu_1,\dots,\mu_n)+\sum_{i=1}^nH(\mu_i),
\end{align*}
we have
$$
\limsup_{N\to\infty}{1\over N}\log\mu_h^{\otimes N}
\Biggl(\prod_{i=1}^n\Delta_R(\mu_i;N,m,\delta)\Biggr)\ge-3\eps.
$$
Here we can take $m$ arbitrarily large and $\delta>0$ arbitrarily small for any given
$\eps>0$. Therefore,
$$
\limsup_{N\to\infty}{1\over N}\log\mu_h^{\otimes N}
\Biggl(\prod_{i=1}^n\Delta_R(\mu_i;N,m,\delta)\Biggr)=0
$$
for all $m\in\bN$ and all $\delta>0$. Since
$$
\mu_h^{\otimes N}\Biggl(\prod_{i=1}^n\Delta_R(\mu_i;N,m,\delta)\Biggr)
\le\mu_{h,i}^{\otimes N}(\Delta_R(\mu_i;N,m,\delta)),
$$
we have
$$
\limsup_{N\to\infty}{1\over N}\log
\mu_{h,i}^{\otimes N}(\Delta_R(\mu_i;N,m,\delta))=0
$$
for all $m\in\bN$, $\delta>0$ and $i=1,\dots,n$. Lemma \ref{L-3.1}\,(b) implies that
$\mu_i=\mu_{h,i}$ for all $i=1,\dots,n$, so (ii) holds.
\end{proof}

\begin{remark}\label{R-3.3}{\rm
Let $\Xi(N)=(\xi_1(N),\dots,\xi_n(N))$ be an approximating sequence for
$(\mu_1,\dots,\mu_n)$. Let $h_1,\dots,h_n\in C_\bR([-R,R])$ and consider $h_i$ as an
element of $C_\bR([-R,R]^n)$ depending on the $i$th variable $x_i$, $1\le i \le n$, so
that $(h_1+\dots+h_n)(\bx)=h_1(x_1)+\dots+h_n(x_n)$ for $\bx=(x_1,\dots,x_n)$. Since
\begin{align*}
&{1\over(N!)^n}\sum_{\sigma_1,\dots,\sigma_n\in S_N}\exp\bigl(
N\kappa_N((h-(h_1+\dots+h_n))(\sigma_1(\xi_1(N)),\dots,\sigma_n(\xi_n(N))))\bigr) \\
&\qquad={1\over(N!)^n}\sum_{\sigma_1,\dots,\sigma_n\in S_N}
\exp\bigl(N\kappa_N(h(\sigma_1(\xi_1(N)),\dots,\sigma_n(\xi_n(N))))\bigr) \\
&\qquad\qquad\times\prod_{i=1}^n\exp\bigl(-N\kappa_N(h_i(\xi_i(N)))\bigr)
\end{align*}
and $\lim_{N\to\infty}\kappa_N(h_i(\xi_i(N)))=\mu_i(h_i)$, it follows that
\begin{align*}
&P_\sym(h-(h_1+\dots+h_n):\mu_1,\dots,\mu_n)+\sum_{i=1}^nP(h_i) \\
&\qquad=P_\sym(h:\mu_1,\dots,\mu_n)+\sum_{i=1}^n(-\mu_i(h_i)+P(h_i)).
\end{align*}
Hence we notice that
\begin{align*}
&P_\sym(h:\mu_1,\dots,\mu_n)+\sum_{i=1}^nH(\mu_i) \\
&\qquad=\inf_{h_1,\dots,h_n}\Biggl\{P_\sym(h-(h_1+\dots+h_n):\mu_1,\dots,\mu_n)
+\sum_{i=1}^nP(h_i)\Biggr\},
\end{align*}
where $h_1+\dots+h_n$ is given as above for $h_1,\dots,h_n\in C_\bR([-R,R])$. In
particular, when $\mu_i$ is the Gibbs measure associated with $h_i$ for $1\le i\le n$, we
have
$$
P_\sym(h:\mu_1,\dots,\mu_n)+\sum_{i=1}^nH(\mu_i)
=P_\sym(h-(h_1+\dots+h_n):\mu_1,\dots,\mu_n)+\sum_{i=1}^nP(h_i).
$$
Hence, if the equivalent conditions (i)--(iii) of Theorem \ref{T-3.2} are satisfied,
then the equality
$$
P(h)=P_\sym(h-(h_1+\dots+h_n):\mu_1,\dots,\mu_n)+\sum_{i=1}^nP(h_i)
$$
holds as well for $h_1,\dots,h_n$ given in (iii).
}\end{remark}

The next lemma is concerned with general relation between $\cI_\sym(\mu)$ and
$I_\sym(\mu)$.

\begin{lemma}\label{L-3.4}
$\cI_\sym(\mu)\le I_\sym(\mu)$ for every $\mu\in\Prob([-R,R]^n)$.
\end{lemma}

\begin{proof}
Let $\mu\in\Prob([-R,R]^n)$ and $\mu_1,\dots,\mu_n$ be the marginals of $\mu$, and
choose an approximating sequence $\Xi(N)=(\xi_1(N),\dots,\xi_n(N))$ for
$(\mu_1,\dots,\mu_n)$. It suffices to prove that
$$
I_\sym(\mu)\ge\mu(p)-P_\sym(p:\mu_1,\dots,\mu_n)
$$
for all real polynomials $p$ of variables $x_1,\dots,x_n$. For any $\eps>0$ there
exist an $m\in\bN$ and a $\delta>0$ such that, for every $N\in\bN$, if
$(\sigma_1,\dots,\sigma_n)\in\Delta_\sym(\mu:\Xi(N);N,m,\delta)$ then
$$
|\kappa_N(p(\sigma_1(\xi_1(N)),\dots,\sigma_n(\xi_n(N))))-\mu(p)|<\eps
$$
so that
$$
e^{N(\mu(p)-\eps)}
<\exp\bigl(N\kappa_N(p(\sigma_1(\xi_1(N)),\dots,\sigma_n(\xi_n(N))))\bigr).
$$
Therefore,
\begin{align*}
&e^{N(\mu(p)-\eps)}{1\over(N!)^n}\,\#\Delta_\sym(\mu:\Xi(N);N,m,\delta) \\
&\qquad\le{1\over(N!)^n}
\sum_{(\sigma_1,\dots,\sigma_n)\in\Delta_\sym(\mu:\Xi(N);N,m,\delta)}
\exp\bigl(N\kappa_N(p(\sigma_1(\xi_1(N)),\dots,\sigma_n(\xi_n(N))))\bigr) \\
&\qquad\le{1\over(N!)^n}\sum_{\sigma_1,\dots,\sigma_n\in S_N}
\exp\bigl(N\kappa_N(p(\sigma_1(\xi_1(N)),\dots,\sigma_n(\xi_n(N))))\bigr),
\end{align*}
which implies that
$$
\mu(p)-\eps-I_\sym(\mu)
\le P_\sym(p:\mu_1,\dots,\mu_n).
$$
This gives the desired inequality since $\eps>0$ is arbitrary.
\end{proof}

The next theorem gives an exact relation between $\cI_\sym(\mu)$ and $H(\mu)$.

\begin{thm}\label{T-3.5}
For every $\mu\in\Prob([-R,R]^n)$ with marginals $\mu_1,\dots,\mu_n\in\Prob([-R,R])$,
$$
H(\mu)=-\cI_\sym(\mu)+\sum_{i=1}^nH(\mu_i).
$$
Moreover, if $H(\mu_i)>-\infty$ for all $i=1,\dots,n$, then
$$
\cI_\sym(\mu)=I_\sym(\mu)=S(\mu,\mu_1\otimes\cdots\otimes\mu_n),
$$
and $\cI_\sym(\mu)=0$ if and only if $\mu=\mu_1\otimes\cdots\otimes\mu_n$, i.e., the
coordinate variables $x_1,\dots,x_n$ are independent with respect to $\mu$.
\end{thm}

\begin{proof}
By \eqref{F-3.1} and Definition \ref{D-2.4}, for every $h\in C_\bR([-R,R]^n)$ we have
\begin{align}
-\mu(h)+P(h)&\ge-\mu(h)+P_\sym(h:\mu_1,\dots,\mu_n)+\sum_{i=1}^nH(\mu_i) \nonumber\\
&\ge-\cI_\sym(\mu)+\sum_{i=1}^nH(\mu_i). \label{F-3.6}
\end{align}
Hence by \eqref{F-1.1}, Lemma \ref{L-3.4} and Theorem \ref{T-1.2} we have
$$
H(\mu)\ge-\cI_\sym(\mu)+\sum_{i=1}^nH(\mu_i)
\ge-I_\sym(\mu)+\sum_{i=1}^nH(\mu_i)=H(\mu)
$$
so that the first assertion is proved. The second assertion immediately follows from the
first and \cite[Corollary 1.7]{HP1}.
\end{proof}

\begin{prop}\label{P-3.6}
Let $h\in C_\bR([-R,R]^n)$ and $\mu\in\Prob([-R,R]^n)$. Let $\mu_1,\dots,\mu_n$ be the
marginals of $\mu$ and $h_1,\dots,h_n$ be as given in (iii) of Theorem \ref{T-3.2}.
Then the following are equivalent:
\begin{itemize}
\item[(i)] $\mu$ is the Gibbs measure associated with $h$;
\item[(ii)] $\mu$ is mutually equilibrium associated with $h$ and $\mu_i$ is the Gibbs
measure associated with $h_i$ for each $i=1,\dots,n$.
\end{itemize}
\end{prop}

\begin{proof}
(i) $\Rightarrow$ (ii).\enspace
Assume that $\mu$ is the Gibbs measure associated with $h$. By \eqref{F-3.6} and
Theorem \ref{T-3.5},
\begin{align*}
H(\mu)&=-\mu(h)+P(h) \\
&\ge-\mu(h)+P_\sym(h:\mu_1,\dots,\mu_n)+\sum_{i=1}^nH(\mu_i) \\
&\ge-\cI_\sym(\mu)+\sum_{i=1}^nH(\mu_i)=H(\mu).
\end{align*}
Moreover, since $\mu_i$ is the $i$th marginal of $\mu=\mu_h$, it follows as in the
proof of Theorem \ref{T-3.2} (see \eqref{F-3.5}) that $\mu_i$ is the Gibbs measure
associated with $h_i$ for $1\le i\le n$. In particular, $H(\mu_i)>-\infty$ for all
$i=1,\dots,n$. Hence
$$
-\cI_\sym(\mu)=-\mu(h)+P_\sym(h:\mu_1,\dots,\mu_n),
$$
that is, $\mu$ is mutually equilibrium associated with $h$.

(ii) $\Rightarrow$ (i).\enspace
Assume (ii). By Theorems \ref{T-3.5} and \ref{T-3.2},
\begin{align*}
H(\mu)&=-\cI_\sym(\mu)+\sum_{i=1}^nH(\mu_i) \\
&=-\mu(h)+P_\sym(h:\mu_1,\dots,\mu_n)+\sum_{i=1}^nH(\mu_i) \\
&=-\mu(h)+P(h)
\end{align*}
so that (i) follows.
\end{proof}

\section{The discrete case}
\setcounter{equation}{0}

In information theory, random variables mostly take values in a discrete set of alphabets
and the basic quantity is the Shannon entropy rather than the Boltzmann-Gibbs entropy.
So the discrete versions of the preceding results in Sections 2 and 3 are of even more
importance, which are presented in this section.

Let $\cX=\{t_1,\dots,t_d\}$ be a finite set of alphabets and consider the $n$-fold product
$\cX^n$. The {\it Shannon entropy} of a probability measure $\mu\in\Prob(\cX)$ is
$$
S(\mu):=-\sum_{t\in\cX}\mu(t)\log\mu(t).
$$
For each sequence $\bx=(x_1,\dots,x_N)\in\cX^N$, the {\it type} of $\bx$ is a probability
measure on $\cX$ given by
$$
\nu_\bx(t):={N_\bx(t)\over N}
\quad\mbox{where\ \ $N_\bx(t):=\#\{j:x_j=t\}$,\ \ $t\in\cX$}.
$$
For each $\mu\in\Prob(\cX)$ (resp.\ $\mu\in\Prob(\cX^n)$) and for each $N\in\bN$ and
$\delta>0$ we denote by $\Delta(\mu;N,\delta)$ the set of all sequences $\bx\in\cX^N$
(resp.\ $\bx\in(\cX^n)^N$) such that $|\nu_\bx(t)-\mu(t)|<\delta$ for all $t\in\cX$
(resp.\ $t\in \cX^n$), that is, $\Delta(\mu;N,\delta)$ is the set
of all $\delta$-typical sequences (with respect to $\mu$). The Shannon entropy has the
following limiting formula:
\begin{equation}\label{F-4.1}
S(\mu)=\lim_{\delta\searrow0}\lim_{N\to\infty}\log\#\Delta(\mu;N,\delta)
\end{equation}
(see \cite{CT,CS} and also \cite[\S2]{HP1} for a concise exposition).

For $N\in\bN$ let $\cX_\le^N$ denote the set of all sequences of length $N$ of the form
$$
\bx=(t_1,\dots,t_1,t_2,\dots,t_2,\dots,t_d,\dots,t_d)
$$
so that $\cX_\le^N$ is regarded as the set of all types from $\cX^N$. The action of $S_N$
on $\cX^N$ is similar to that on $\bR^N$ given in Definition \ref{D-1.1}.

\begin{definition}\label{D-4.1}{\rm
Let $\mu\in\Prob(\cX^n)$ and $\mu_i\in\Prob(\cX)$ be the $i$th marginal of $\mu$ for
$1\le i\le n$. Choose an approximating sequence $\Xi(N)=(\xi_1(N),\dots,\xi_n(N))$,
$N\in\bN$, for $(\mu_1,\dots,\mu_n)$, that is, $\xi_i(N)\in\cX_\le^N$ and
$\nu_{\xi_i(N)}(t)\to\mu_i(t)$ as $N\to\infty$ for all $t\in\cX$ and $i=1,\dots,n$.
For each $N\in\bN$ and $\delta>0$ we define $\Delta_\sym(\mu:\Xi(N);N,\delta)$ to be the
set of all $(\sigma_1,\dots,\sigma_n)\in S_N^n$ such that
$$
(\sigma_1(\xi_1(N)),\dots,\sigma_n(\xi_n(N)))\in\Delta(\mu;N,\delta).
$$
We define
$$
I_\sym(\mu):=-\lim_{\delta\searrow0}\limsup_{N\to\infty}{1\over N}
\log\gamma_{S_N}^{\otimes n}(\Delta_\sym(\mu:\Xi(N);N,\delta))
$$
and $\overline I_\sym(\mu)$ by replacing $\limsup$ by $\liminf$. See \cite[Lemma 2.4]{HP1}
for the independence of the choice of $\Xi(N)$ for $I_\sym(\mu)$ and
$\overline I_\sym(\mu)$ as well as their equivalent definitions.
}\end{definition}

The two quantities $I_\sym(\mu)$ and $\overline I_\sym(\mu)$ are equal and connected to
$S(\mu)$ as follows.

\begin{thm}\label{T-4.2}{\rm(\cite[Theorem 2.5]{HP1})}\quad
For every $\mu\in\Prob(\cX^n)$ with marginals $\mu_1,\dots,\mu_n\in\Prob(\cX)$,
$$
I_\sym(\mu)=\overline I_\sym(\mu)=-S(\mu)+\sum_{i=1}^nS(\mu_i).
$$
\end{thm}

We denote by $C_\bR(\cX^n)$ the real Banach space of real functions on $\cX^n$ with the
norm $\|f\|:=\max\{|f(\bx)|:\bx\in\cX^n\}$.

\begin{definition}\label{D-4.3}{\rm
Let $\mu_1,\dots,\mu_n\in\Prob(\cX)$ and choose an approximating sequence
$\Xi(N)=(\xi_1(N),\dots,\xi_n(N))$ for $(\mu_1,\dots,\mu_n)$ as given in Definition
\ref{D-4.1}. For each $h\in C_\bR(\cX^n)$ and $\bx_i\in\cX^N$, $1\le i\le n$, define
$h(\bx_1,\dots,\bx_n)$ and $\kappa_N(h(\bx_1,\dots,\bx_n))$ in the same manner as in
\eqref{F-2.1} and \eqref{F-2.2} so that
$$
\kappa_N(h(\bx_1,\dots,\bx_n))=\sum_{t\in\cX^n}h(t)\nu_{(\bx_1,\dots,\bx_n)}(t)
$$
for $(\bx_1,\dots,\bx_n)$ regarded as a sequence in $(\cX^n)^N$. We define the
{\it mutual pressure} of $h$ with respect to $(\mu_1,\dots,\mu_n)$ to be
\begin{align*}
&P_\sym(h:\mu_1,\dots,\mu_n) \\
&\quad:=\limsup_{N\to\infty}{1\over N}\log
\Biggl[{1\over(N!)^n}\sum_{\sigma_1,\dots,\sigma_n\in S_N}
\exp\bigl(N\kappa_N(h(\sigma_1(\xi_1(N)),\dots,\sigma_n(\xi_n(N))))\bigr)\Biggr].
\end{align*}
Moreover, for each $\mu\in\Prob(\cX^n)$ with marginals $\mu_1,\dots,\mu_n\in\Prob(\cX)$
we define
$$
\cI_\sym(\mu):=\sup\{\mu(h)-P_\sym(h:\mu_1,\dots,\mu_n):h\in C_\bR(\cX^n)\},
$$
and we say that $\mu$ is {\it mutually equilibrium} associated with $h$ if the equality
$$
\cI_\sym(\mu)=\mu(h)-P_\sym(h:\mu_1,\dots,\mu_n)
$$
holds.
}\end{definition}

Then all the results in Section 2 are valid in this discrete setting as well. To see this,
it is convenient to reduce the discrete case to a special case of the continuous case of
Section 2 in the following way. Choose $d$ points $\hat t_1<\hat t_2<\dots<\hat t_d$ in
$[-R,R]$ corresponding to $t_1,t_2,\dots,t_d$ in $\cX$. For each $\mu\in\Prob(\cX^n)$ with
marginals $\mu_1,\dots,\mu_n$ we have the corresponding (atomic) probability measure
$\hat\mu\in\Prob([-R,R]^n)$ given by
$$
\hat\mu:=\sum_{\bx\in\cX^n}\mu(\bx)\delta_{\hat\bx},
$$
and similarly $\hat\mu_1,\dots,\hat\mu_n\in\Prob([-R,R])$, $1\le i\le n$. Then the
marginals of $\hat\mu$ are $\hat\mu_1,\dots,\hat\mu_n$. For each approximating sequence
$(\xi_1(N),\dots,\xi_n(N))$ for $(\mu_1,\dots,\mu_n)$ we have the corresponding
$\hat\xi_i(N)\in[-R,R]_\le^N$, $1\le i\le n$. Since
$$
\kappa_N(\hat\xi_i(N)^k)=\sum_{t\in\cX}\hat t^k\nu_{\xi_i(N)}(t)
\longrightarrow\sum_{t\in\cX}\hat t^k\mu_i(t)=\int x^k\,d\hat\mu_i(x)
\quad\mbox{as $N\to\infty$}
$$
for all $k\in\bN$, it follows that $(\hat\xi_1(N),\dots,\hat\xi_n(N))$ is an approximating
sequence for $(\hat\mu_1,\dots,\hat\mu_n)$. For each $h\in C_\bR(\cX^n)$ choose an
$\hat h\in C_\bR([-R,R]^n)$ such that $\hat h(\hat\bx)=h(\bx)$ for all $\bx\in\cX^n$. Then
we notice that $P_\sym(h:\mu_1,\dots,\mu_n)$ in Definition \ref{D-4.3} is equal to
$P_\sym(\hat h:\hat\mu_1,\dots,\hat\mu_n)$ defined in Definition \ref{D-2.1}, and that
$\cI_\sym(\mu)$ in Definition \ref{D-4.3} is equal to $\cI_\sym(\hat\mu)$ defined in
Definition \ref{D-2.4}. Upon these considerations it is rather straightforward to show
the discrete versions of the results in Section 2. For example, for $h,h'\in C_\bR(\cX^n)$
choose $\hat h,\hat g\in C_\bR([-R,R]^n)$ such that $\hat h|_{\cX^n}=h$,
$\|\hat h\|=\|h\|$, $\hat g|_{\cX^n}=h-h'$ and $\|\hat g\|=\|h-h'\|$, and define
$\hat h':=\hat h-\hat g$. Then $\hat h'|_{\cX^n}=h'$ and $\|\hat h-\hat h'\|=\|h-h'\|$.
Hence the discrete version of Proposition \ref{P-2.3}\,(3) is seen as follows:
\begin{align*}
&|P_\sym(h:\mu_1,\dots,\mu_n)-P_\sym(h':\mu_1,\dots,\mu_n)| \\
&\qquad=|P_\sym(\hat h:\hat\mu_1,\dots,\hat\mu_n)
-P_\sym(\hat h':\hat\mu_1,\dots,\hat\mu_n)| \\
&\qquad\le\|\hat h-\hat h'\|=\|h-h'\|.
\end{align*}
Also, the discrete version of Proposition \ref{P-2.5} is seen as follows:
\begin{align*}
P_\sym(h:\mu_1,\dots,\mu_n)
&=P_\sym(\hat h:\hat\mu_1,\dots,\hat\mu_n) \\
&=\max\{\lambda(\hat h)-\cI_\sym(\lambda):
\lambda\in\Prob_{\hat\mu_1,\dots,\hat\mu_n}([-R,R]^n)\} \\
&=\max\{\mu(h)-\cI_\sym(\mu):\mu\in\Prob_{\mu_1,\dots,\mu_n}(\cX^n)\}
\end{align*}
since $\Prob_{\hat\mu_1,\dots,\hat\mu_n}([-R,R]^n)
=\{\hat\mu:\mu\in\Prob_{\mu_1,\dots,\mu_n}(\cX^n)\}$.

Now let us show the discrete version of Theorem \ref{T-3.2}. Although the proof is
essentially same as that of Theorem \ref{T-3.2}, some non-trivial modifications are
necessary due to the difference between the Shannon and Boltzmann-Gibbs entropies.

\begin{thm}\label{T-4.4}
For every $h\in C_\bR(\cX^n)$ and every $\mu_1,\dots,\mu_n\in\Prob(\cX)$,
\begin{equation}\label{F-4.2}
P(h)\ge P_\sym(h:\mu_1,\dots,\mu_n)+\sum_{i=1}^nS(\mu_i),
\end{equation}
and the following conditions are equivalent:
\begin{itemize}
\item[(i)] $P(h)=P_\sym(h:\mu_1,\dots,\mu_n)+\sum_{i=1}^nS(\mu_i)$;
\item[(ii)] $\mu_1,\dots,\mu_n$ are the marginals of the Gibbs measure $\mu_h$ associated
with $h$ given by
\begin{equation}\label{F-4.3}
\mu_h(\bx):={1\over Z_h}\,e^{h(\bx)},\quad\bx\in\cX^n
\quad\mbox{with}\quad Z_h:=\sum_{\bx\in\cX^n}e^{h(\bx)};
\end{equation}
\item[(iii)] for each $i=1,\dots,n$, $\mu_i$ is the Gibbs measure associated with
$h_i\in C_\bR(\cX)$ defined by
$$
h_i(x):=\log\sum_{x_1,\dots,x_{i-1},x_{i+1},\dots,x_n\in\cX}
e^{h(x_1,\dots,x_{i-1},x,x_{i+1},\dots,x_n)}
$$
for $x\in\cX$.
\end{itemize}
\end{thm}

\begin{proof}
Let $\mu_h$ be the Gibbs measure given in \eqref{F-4.3}, and let
$(\xi_1(N),\dots,\xi_n(N))$ be an approximating sequence for $(\mu_1,\dots,\mu_n)$.
For any $\eps>0$ one can choose a $\delta>0$ such that for every $i=1,\dots,n$ and every
$p\in\Prob(\cX)$, if $|p(t)-\mu_i(t)|<\delta$ for all $t\in\cX$, then
$|S(p)-S(\mu_i)|<\eps/n$. This means that for each $N\in\bN$ and $i=1,\dots,n$, one has
$|S(\nu_\bx)-S(\mu_i)|<\eps/n$ whenever $\bx\in\Delta(\mu_i;N,\delta)$. Furthermore, when
$\delta>0$ is small enough, one can find an $N_0\in\bN$ such that, for every $N\ge N_0$,
if $\bx_i\in\Delta(\mu_i;N,\delta)\cap\cX_\le^N$ for $1\le i\le n$, then
\begin{equation}\label{F-4.4}
|\kappa_N(h(\sigma_1(\bx_1),\dots,\sigma_n(\bx_n))
-\kappa_N(h(\sigma_1(\xi_1(N)),\dots,\sigma_n(\xi_n(N))))|<\eps
\end{equation}
for all $(\sigma_1,\dots,\sigma_n)\in S_N$.

For each sequence $(N_1,\dots,N_d)$ of integers $N_l\ge0$ with $\sum_{l=1}^dN_l=N$, let
$S(N_1,\dots,N_d)$ denote the subgroups of $S_N$ consisting of products of permutations of
$\{1,\dots,N_1\}$, $\{N_1+1,\dots,N_1+N_2\}$, $\dots$, $\{N_1+\dots+N_{d-1}+1,\dots,N\}$,
and let $S_N/S(N_1,\dots,N_d)$ be the set of left cosets of $S(N_1,\dots,N_d)$. For each
$\bx\in\cX_\le^N$ we write $S_\bx$ for $S(N_\bx(t_1),\dots,N_\bx(t_d))$. For $N\in\bN$ it
then follows that
\begin{align}
&Z_h^N\mu_h^{\otimes n}\Biggl(\prod_{i=1}^n\Delta(\mu_i;N,\delta)\Biggr) \nonumber\\
&\qquad=\sum_{(\bx_1,\dots,\bx_n)\in\prod_{i=1}^n\Delta(\mu_i;N,\delta)}
\exp\bigl(N\kappa_N(h(\bx_1,\dots,\bx_n))\bigr) \nonumber\\
&\qquad=\sum_{(\bx_1,\dots,\bx_n)\in\prod_{i=1}^n
\bigl(\Delta(\mu_i;N,\delta)\cap\cX_\le^N\bigr)} \nonumber\\
&\hskip3cm\sum_{([\sigma_1],\dots,[\sigma_n])\in(S_N/S_{\bx_1},\dots,S_N/S_{\bx_n})}
\exp\bigl(N\kappa_N(h(\sigma_1(\bx_1),\dots,\sigma_n(\bx_n)))\bigr), \label{F-4.5}
\end{align}
where $\prod_{i=1}^n\Delta(\mu_i;N,\delta)$ in the left-hand side is regarded as a subset
of $(\cX^n)^N$ in the same manner as in the beginning of the proof of Theorem \ref{T-3.2},
and $[\sigma_i]$ denotes the coset of $S_{\bx_i}$ containing $\sigma_i$. Moreover we have
\begin{align}
&\sum_{\sigma_1,\dots,\sigma_n\in S_N}
\exp\bigl(N\kappa_N(h(\sigma_1(\bx_1),\dots,\sigma_n(\bx_n)))\bigr) \nonumber\\
&\qquad=\sum_{([\sigma_1],\dots,[\sigma_n])\in(S_N/S_{\bx_1},\dots,S_N/S_{\bx_n})}
\Biggl(\prod_{i=1}^n\prod_{l=1}^dN_{\bx_i}(t_l)!\Biggr)
\exp\bigl(N\kappa_N(h(\sigma_1(\bx_1),\dots,\sigma_n(\bx_n)))\bigr). \label{F-4.6}
\end{align}
For each $i=1,\dots,n$ and for any $\bx\in\cX^N$, the Stirling formula implies that
\begin{align*}
&{1\over N}\sum_{l=1}^d\log N_\bx(t_l)!-{1\over N}\log N! \\
&\qquad={1\over N}\sum_{l=1}^d
\biggl(N_\bx(t_l)\log N_\bx(t_l)-N_\bx(t_l)+{1\over2}\log N_\bx(t_l)+O(1)\biggr) \\
&\qquad\qquad-{1\over N}\biggl(N\log N-N+{1\over2}\log N+O(1)\biggr) \\
&\qquad=\sum_{l=1}^d{N_\bx(t_l)\over N}\log N_\bx(t_l)-\log N+o(1) \\
&\qquad=-S(\nu_\bx)+o(1)\quad\mbox{as $N\to\infty$},
\end{align*}
where $o(1)$ as $N\to\infty$ is uniform for $\bx\in\cX^N$. Thanks to the above choice of
$\delta>0$, for every $(\bx_1,\dots,\bx_n)\in\prod_{i=1}^n\Delta(\mu_i;N,\delta)$ we have
\begin{align}
&\exp\Biggl[N\Biggl(-\sum_{i=1}^nS(\mu_i)-\eps+o(1)\Biggr)\Biggr] \nonumber\\
&\qquad\le{\prod_{i=1}^n\prod_{l=1}^dN_{\bx_i}(t_i)\over(N!)^n}
\le\exp\Biggl[N\Biggl(-\sum_{i=1}^nS(\mu_i)+\eps+o(1)\Biggr)\Biggr]
\quad\mbox{as $N\to\infty$}, \label{F-4.7}
\end{align}
where $o(1)$ is uniform for $(\bx_1,\dots,\bx_n)\in\prod_{i=1}^n\Delta(\mu_i;N,\delta)$.

Combining \eqref{F-4.5}--\eqref{F-4.7} yields
\begin{align*}
&Z_h^N\mu_h^{\otimes N}\Biggl(\prod_{i=1}^n\Delta(\mu_i;N,\delta)\Biggr) \\
&\quad\ge\sum_{(\bx_1,\dots,\bx_n)\in\prod_{i=1}^n
\bigl(\Delta(\mu_i;N,\delta)\cap\cX_\le^N\bigr)}
{1\over(N!)^n}\sum_{\sigma_1,\dots,\sigma_n\in S_N}
\exp\bigl(N\kappa_N(h(\sigma_1(\bx_1),\dots,\sigma_n(\bx_n)))\bigr) \\
&\hskip8cm\times\exp\Biggl[N\Biggl(\sum_{i=1}^nS(\mu_i)-\eps+o(1)\Biggr)\Biggr]
\end{align*}
and the reverse inequality with $+\eps$ in place of $-\eps$ in the last term. By this
together with \eqref{F-4.4} we obtain
\begin{align}
Z_n^N
&\ge Z_h^N\mu_h^{\otimes N}\Biggl(\prod_{i=1}^n\Delta(\mu_i;N,\delta)\Biggr) \nonumber\\
&\ge e^{-2N\eps}{1\over(N!)^n}\sum_{\sigma_1\dots,\sigma_n\in S_N}
\exp\bigl(N\kappa_N(h(\sigma_1(\xi_1(N)),\dots,\sigma_n(\xi_n(N))))\bigr) \nonumber\\
&\hskip3cm\times\prod_{i=1}^n\#\bigl(\Delta(\mu_i;N,\delta)\cap\cX_\le^N\bigr)
\cdot\exp\Biggl[N\Biggl(\sum_{i=1}^nS(\mu_i)+o(1)\Biggr)\Biggr] \label{F-4.8}
\end{align}
and
\begin{align}
&Z_h^N\mu_h^{\otimes N}\Biggl(\prod_{i=1}^n\Delta(\mu_i;N,\delta)\Biggr) \nonumber\\
&\qquad\le e^{2N\eps}{1\over(N!)^n}\sum_{\sigma_1\dots,\sigma_n\in S_N}
\exp\bigl(N\kappa_N(h(\sigma_1(\xi_1(N)),\dots,\sigma_n(\xi_n(N))))\bigr) \nonumber\\
&\hskip3.5cm\times\prod_{i=1}^n\#\bigl(\Delta(\mu_i;N,\delta)\cap\cX_\le^N\bigr)
\cdot\exp\Biggl[N\Biggl(\sum_{i=1}^nS(\mu_i)+o(1)\Biggr)\Biggr] \label{F-4.9}
\end{align}
for all $N\ge N_0$. Furthermore, since
$$
\Delta(\mu_i;N,\delta)=\bigl\{\sigma(\bx):
\bx\in\Delta(\mu_i;N,\delta)\cap\cX_\le^N,\,[\sigma]\in S_N/S_{\bx}\bigr\}
$$
so that
$$
\#\Delta(\mu_i;N,\delta)=\sum_{\bx\in\Delta(\mu_i;N,\delta)\cap\cX_\le^N}
{N!\over\prod_{l=1}^dN_\bx(t_l)!},
$$
we have as inequalities in \eqref{F-4.7}
\begin{align*}
&\#\bigl(\Delta(\mu_i;N,\delta)\cap\cX_\le^N\bigr)
\cdot\exp\biggl[N\biggl(S(\mu_i)-{\eps\over n}+o(1)\biggr)\biggr] \\
&\quad\le\#\Delta(\mu_i;N,\delta)
\le\#\bigl(\Delta(\mu_i;N,\delta)\cap\cX_\le^N\bigr)
\cdot\exp\biggl[N\biggl(S(\mu_i)+{\eps\over n}+o(1)\biggr)\biggr]
\quad\mbox{as $N\to\infty$}.
\end{align*}
This and \eqref{F-4.1} imply that
\begin{align}
-{\eps\over n}&\le\liminf_{N\to\infty}{1\over N}\log
\#\bigl(\Delta(\mu_i;N,\delta)\cap\cX_\le^N\bigr) \nonumber\\
&\le\limsup_{N\to\infty}{1\over N}\log
\#\bigl(\Delta(\mu_i;N,\delta)\cap\cX_\le^N\bigr)\le{\eps\over n}. \label{F-4.10}
\end{align}
It follows from \eqref{F-4.8} that
\begin{align*}
P(h)&={1\over N}\log Z_h^N \\
&\ge-2\eps+{1\over N}\log\Biggl[{1\over(N!)^n}\sum_{\sigma_1,\dots,\sigma_n\in S_N}
\exp\bigl(N\kappa_N(h(\sigma_1(\xi_1(N)),\dots,\sigma_n(\xi_n(N))))\bigr)\Biggr] \\
&\qquad+\sum_{i=1}^n{1\over N}\log\#\bigl(\Delta(\mu_i;N,\delta)\cap\cX_\le^N\bigr)
+\sum_{i=1}^nS(\mu_i)+o(1)\quad\mbox{as $N\to\infty$},
\end{align*}
which implies that
$$
P(h)\ge-3\eps+P_\sym(h:\mu_1,\dots,\mu_n)+\sum_{i=1}^nS(\mu_i)
$$
thanks to \eqref{F-4.10}. Hence inequality \eqref{F-4.2} follows since $\eps>0$ is
arbitrary.

To prove the equivalence of (i)--(iii), let $\mu_{h,i}$ be the $i$th marginal of $\mu_h$.
Then it follows that $\mu_{h,i}(x)=Z_h^{-1}e^{h_i(x)}$ and so $\mu_{h,i}$ is the Gibbs
measure associated with $h_i$ for $1\le i\le n$. Hence (ii) $\Leftrightarrow$ (iii)
follows. Assume (ii), i.e., that $\mu_i=\mu_{h,i}$ for all $i=1,\dots,n$. Since we have
$\lim_{N\to\infty}\mu_i^{\otimes N}(\Delta(\mu_i;N,\delta))=1$ based on the Sanov theorem
as in Lemma \ref{L-3.1}\,(a), it follows that
$$
\lim_{N\to\infty}\mu_h^{\otimes N}\Biggl(\prod_{i=1}^n\Delta(\mu_i;N,\delta)\Biggr)=1
$$
as in the proof of (ii) $\Rightarrow$ (i) of Theorem \ref{T-3.2}. Combining this with
\eqref{F-4.9} and \eqref{F-4.10} yields
$$
P(h)\le3\eps+P_\sym(h:\mu_1,\dots,\mu_n)+\sum_{i=1}^nS(\mu_i),
$$
which implies equality in (i). Conversely, assume (i). Then \eqref{F-4.8} and
\eqref{F-4.10} imply that
$$
\limsup_{N\to\infty}{1\over N}\log\mu_h^{\otimes N}
\Biggl(\prod_{i=1}^n\Delta(\mu_i;N,\delta)\Biggr)\ge-3\eps.
$$
The same reasoning as in the last part of the proof of Theorem \ref{T-3.2} gives
$$
\limsup_{N\to\infty}{1\over N}\log\mu_{h,i}^{\otimes N}(\Delta(\mu_i;N,\delta))=0
$$
for all $\delta>0$ and $i=1,\dots,n$. Since we have a result similar to Lemma
\ref{L-3.1}\,(b) in the present discrete situation, it follows that $\mu_i=\mu_{h,i}$ for
all $i=1,\dots,n$, and so (ii) holds.
\end{proof}

The next theorem and proposition are the discrete versions of Theorem \ref{T-3.5} and
Proposition \ref{P-3.6}. Since their proofs based on Theorems \ref{T-4.2} and \ref{T-4.4}
are similar to those in Section 3, we omit the details. Here note only that
$\cI_\sym(\mu)\le I_\sym(\mu)$ for every $\mu\in\Prob(\cX^n)$ can be shown similarly to
the proof of Lemma \ref{L-3.4} or by the same reasoning as given after Definition
\ref{D-4.3}, and that the Legendre transform expression as in \eqref{F-1.1}
$$
S(\mu)=\inf\{-\mu(h)+P(h):h\in C_\bR(\cX^n)\}
$$
is valid for every $\mu\in\Prob(\cX^n)$.

\begin{thm}\label{T-4.5}
For every $\mu\in\Prob(\cX^n)$ with marginals $\mu_1,\dots,\mu_n\in\Prob(\cX)$,
$$
\cI_\sym(\mu)=I_\sym(\mu)=-S(\mu)+\sum_{i=1}^nS(\mu_i).
$$
\end{thm}

\begin{prop}\label{P-4.6}
Let $h\in C_\bR(\cX^n)$ and $\mu\in\Prob(\cX^n)$. Let $\mu_1,\dots,\mu_n$ be the marginals
of $\mu$ and $h_1,\dots,h_n$ be as given in (iii) of Theorem \ref{T-4.4}. Then the
following are equivalent:
\begin{itemize}
\item[(i)] $\mu$ is Gibbs measure associated with $h$;
\item[(ii)] $\mu$ is mutually equilibrium associated with $h$ and $\mu_i$ is the Gibbs
measure associated with $h_i$ for each $i=1,\dots,n$.
\end{itemize}
\end{prop}


\begin{thebibliography}{99}

\bibitem{CT}
T. M. Cover and J. A. Thomas,
{\it Elements of Information Theory}, 2nd ed., Wiley-Interscience, Hoboken, NJ, 2006.

\bibitem{CS}
I. Csisz\'ar and P. C. Shields, {\it Information Theory and Statistics: A Tutorial},
in {\it Foundations and Trends in Communications and Information Theory},
Vol. 1, No. 4 (2004), 417-528, Now Publishers.

\bibitem{DZ}
A. Dembo and O. Zeitouni,
{\it Large deviations techniques and applications}, Second edition.
Applications of Mathematics, Vol. 38, Springer-Verlag, New York, 1998.

\bibitem{HP}
F. Hiai and D. Petz,
{\it The Semicircle Law, Free Random Variables and Entropy},
Mathematical Surveys and Monographs, Vol. 77, Amer. Math. Soc., Providence, 2000.

\bibitem{HP1}
F. Hiai and D. Petz,
{\it A new approach to mutual information},
in {\it Noncommutative Harmonic Analysis with Applications to Probability},
M. Bo{\.z}ejko et al. (eds.), Banach Center Publications, Vol.~78, 2007, pp.~151--164.

\bibitem{V2}
D. Voiculescu,
{\it The analogues of entropy and of Fisher's information measure in free probability
theory, II}, Invent. Math. 118 (1994), 411--440.

\end{thebibliography}
\end{document}